\documentclass{amsart}
\usepackage{amssymb}
\usepackage{mathrsfs}
\usepackage{graphicx}
\usepackage{comment}

\newcommand{\ve}{\varepsilon}

\newcommand{\nni}{\noindent}
\newcommand{\be}{\begin{equation}}
\newcommand{\ee}{\end{equation}}
\newcommand{\ba}{\begin{align}}
\newcommand{\ea}{\end{align}}

\newtheorem{theorem}{Theorem}[section]

\newtheorem{lemma}{Lemma}[section]

{\begin{list}{}{%
\settowidth{\labelwidth}{\textsf{{\it #1.}}}%
\setlength{\labelsep}{4mm}%
\setlength{\leftmargin}{\labelwidth}%
\addtolength{\leftmargin}{\labelsep}%
}}%
{\end{list}}

\def\om{\omega}
\def\beq{\begin{equation}}\def\enq{\end{equation}}

{\begin{list}{}{%
\settowidth{\labelwidth}{\textsf{{\it #1.}}}%
\setlength{\labelsep}{2mm}%
\setlength{\leftmargin}{\labelwidth}%
\addtolength{\leftmargin}{\labelsep}%
\addtolength{\leftmargin}{4mm}%
\setlength{\itemsep}{6pt}%
\setlength{\listparindent}{0pt}%
\setlength{\topsep}{3pt}%
}}%

\title[Integer group determinants ]{The integer group determinants for the abelian groups of order 18}

\author[B. Paudel]{Bishnu Paudel}
\address{Mathematics and Statistics Department\\
         University of Minnesota Duluth\\
      Duluth, MN 55812, USA}
\email{bpaudel@d.umn.edu}

\author[C. Pinner]{Chris Pinner}
\address{ Department of Mathematics\\
         Kansas State University\\
         Manhattan, KS 66506, USA}
\email{pinner@math.ksu.edu}


\date{\today}
\begin{document}

\begin{abstract}
We obtain a complete description of the integer group determinants for $\mathbb Z_{18}$  (these are the $18\times18$ circulant determinants with integer entries) and $\mathbb Z_3 \times \mathbb Z_6$, the two abelian groups of order 18. This completes the groups of order less than 20.

\end{abstract}

\maketitle

\section{Introduction} 
At the 1977 meeting of the AMS in Hayward, California, Olga Taussky-Todd  \cite{TausskyTodd} asked whether one could determine the set of integer values taken by a group  determinant when the entries are all integers, with particular interest in the case of cyclic groups $\mathbb Z_n$ (where the group determinants are the $n\times n$ circulant determinants). As shown by Newman \cite{Newman1} and Laquer \cite{Laquer}, this is straightforward when $G=\mathbb Z_p$ or $\mathbb Z_{2p}$,  $p$ a prime.
Similar results hold for $D_{2p}$ and $D_{4p}$, the Dihedral groups of order $2p$ and $4p$
(see \cite{Mahoney} and \cite{dihedral}). In general though, the problem quickly becomes complicated (even in the case of $\mathbb Z_{p^2}$ once $p\geq 5$, see \cite{Newman2}). A complete description was obtained for groups $G$ with $|G|\leq 14$ 
in \cite{smallgps}, extended to 15 in \cite{bishnu1}. A series of papers \cite{Humb,Humb2,dihedral,ZnxH,Q16,Yamaguchi1,Yamaguchi2,Yamaguchi3,Yamaguchi4,
Yamaguchi5,Yamaguchi6,Yamaguchi9}  dealt with the 14 groups
of order 16 (see the last paper for a nice overview). In
\cite{nonabelian}  we considered the non-abelian groups of order 18. Here we give the set of integer group determinants for $\mathbb Z_{18}$ and $\mathbb Z_3 \times \mathbb Z_6$, the abelian groups of order 18. This resolves the Taussky-Todd integer group determinant
problem for all $|G|<20$ (3 groups of order 20 remain unresolved).

We shall think of the integer group determinant as being defined on elements of the Group Ring $\mathbb Z [G]$:
$$ D\left(\sum_{g\in G} a_g \: g \right) = \det\left( a_{gh^{-1}} \right), $$
where $g$ indexes rows and $h$ columns. We write $S(G)$ for the set of integer group determinants for $G$. 
We recall that $S(G)$ is closed under multiplication; for $\alpha,\beta\in \mathbb Z[G]$ we have $D(\alpha\beta)=D(\alpha)D(\beta)$. We also note Newman and Laquer's cyclic 
results
\be \label{Newman} \gcd(m,n)=1\Rightarrow m\in S(\mathbb Z_n),\quad n^2\mathbb Z\subset S(\mathbb Z_n), \ee
and the divisibility restrictions on an integer group determinant $D$ in $S(\mathbb Z_n)$:
\be \label{cyclicdiv} p^{\alpha}\parallel n, \; p\mid D \; \Rightarrow \; p^{\alpha+1} \mid D. \ee
Conditions \eqref{cyclicdiv} are sufficient as well as necessary for $G=\mathbb Z_p$, $\mathbb Z_{2p}$ and $\mathbb Z_9$.

\section{\Large The Integer Group Determinants for $\mathbb Z_3 \times \mathbb Z_6$} 

\vskip0.1in
When calculating the group determinant, we shall think of $\mathbb Z_3\times \mathbb Z_6$ as 
$$G=\mathbb Z_2 \times \mathbb Z_3 \times \mathbb Z_3=\langle X,Y,Z \; :\; X^2=Y^3=Z^3=1\rangle. $$ 
As observed by Dedekind, in the case of an abelian group,  the group determinant
factors nicely  into linear factors; for example, see \cite{Conrad} or \cite{book}.
For an element
 $$F(X,Y,Z)=\sum_{i=0}^1\sum_{j=0}^2\sum_{k=0}^2 a_{ijk}X^iY^jZ^k$$
in $\mathbb Z[G]$,
we have the corresponding group determinant
$$D=\prod_{x^2=1}\prod_{y^3=1} \prod_{z^3=1} F(x,y,z),$$ 
which takes the form
$D=D_1D_2,$ where $D_1$ and $D_2$ are the $\mathbb Z_3 \times \mathbb Z_3$ 
determinants of $F(1,y,z)$ and $F(-1,y,z)$:
$$ D_1=\prod_{y^3=1} \prod_{z^3=1} F(1,y,z),\quad  D_2=\prod_{y^3=1} \prod_{z^3=1} F(-1,y,z). $$

 From
$$ S(\mathbb Z_3 \times \mathbb Z_3)=\{ 3^6m \text{ or } 9m\pm 1 \; : \; m\in \mathbb Z\}, $$
and $F(1,y,z)\equiv F(-1,y,z)$ mod 2, we get the restrictions
\be \label{restrictions}  D_1\equiv D_2 \bmod 2,\quad 3^6\mid D_i \text{ or } D_i\equiv \pm 1 \bmod 9. \ee

For the $3^\ell\parallel D$ with $\ell \neq 7,$ we get all $D=D_1D_2$ satisfying \eqref{restrictions}.

\begin{theorem}\label{not7}
For $G=\mathbb Z_3\times \mathbb Z_6$, the integer group determinants $D$ with $3^{\ell}\parallel D$, $\ell \neq 7,$ take the following forms.

The odd values coprime to 3 are the $\pm (1+18m)$, $m\in \mathbb Z$. 

The even values coprime to 3 are the $\pm 4(4+9m)(4+9k)$, $m,k\in \mathbb Z$.

The odd multiples of 3 are the $\pm 3^6(1+6m)$ or $3^8(1+2m)$, $m\in \mathbb Z$.

The even multiples of 3 are the $\pm 2^2\cdot 3^6 (1+3m)(4+9k)$ or $2^2\cdot 3^8m(4+9k)$, $m,k\in \mathbb Z$ or $2^2\cdot 3^{12}m$, $m\in \mathbb Z$.

\end{theorem}

This still  leaves the $D=3^7m$, $\gcd(m,3)=1$. In the case of $D$ odd, we can obtain all values which contain certain prime factors $p\equiv 1$ mod 3. We observe that all primes $p\equiv 1$ mod 3 
are norms of elements in $\mathbb Z [\om]$, $\om :=e^{2\pi i/3},$
$$N(a+b\om)=(a+b \om)(a+b\om^2)=a^2-ab+b^2, \quad a,b\in \mathbb Z. $$
Moreover, as a consequence of Lemma \ref{evenodd} below, we can assume that $6\mid ab$.

\begin{lemma} If $p\equiv 1$ mod 3, then
\be  \label{types} p=N(a+ 3b \om ),\quad a,b\in \mathbb Z, \;\; 2\mid ab. \ee
\end{lemma}

We divide the $p\equiv 1$ mod 3 into two types. Type 1:
\begin{align*} & 7, 13, 19, 37, 61, 67, 73, 79, 97, 103, 139, 151, 163, 181, 193, 199, 211, 241, 271, 313, 331,\\
& 337,349, 367, 373, 379, 409, 421, 463, 487, 523, 541, 547, 571, 577, 607, 613, 619, 631, 661,\\ & 673, 709, 751, 757, 769, 787, 823, 829, 853, 859, 877, 883, 907, 937,\ldots 
\end{align*}
where $2\nmid b$ in \eqref{types}, that is, since $N(a+b\om)=N(-a-b\om)$,
\be \label{Type1} p=N( 2(1+3B) + 3(1+2A)\om),\quad A,B\in \mathbb Z, \ee
and Type 2 
\begin{align} \label{Type2} &  31, 43, 109, 127, 157, 223, 229, 277, 283, 307, 397, 433, 439, 457, 499, 601, 643, 691, 727, \\\nonumber 
& 733, 739, 811, 919, 997, 1021, 1051, 1069, 1093, 1327, 1399, 1423, 1459, 1471, 1579, 1597, \\ \nonumber & 1627, 1657, 1699, 1723, 1753, 1777, 1789, 1801, 1831, 1933, 1999, 2017,\ldots  
\end{align}
where $2\mid b$, that is
$$ p=N(1+6A + 6B\om ),\quad A,B\in \mathbb Z. $$
Type 1 is OEIS sequence A040034, the primes of the form $4x^2+2xy+7y^2$. Type 2 
 is OEIS sequence A014752, the primes  of the form $x^2+27y^2$. This can be seen by writing $N(2A+3B\om)=4(A-B)^2+2(A-B)B+7B^2$ and 
$N(A+6B\om)=(A-3B)^2+27B^2$. Type 1 can also be characterized as the primes where $2$ is not a cube mod $p$ and Type 2 as  the primes $p\equiv 1\bmod 3$ where $2$ is a cube mod $p$.
Note, a $p\equiv 1\bmod 3$ cannot be both a Type 1 and a Type 2 (see Lemma \ref{properties} below).

In the case of $3^7 \parallel D$, $D$  odd, we show that $D$ must contain a Type 1 prime and that we can obtain all values which contain a Type 1 prime. For the $3^7\parallel D$,  $D$ even, we achieve all $D$ with $2^6\mid D,$ 
but only certain values with $2^{\ell}\parallel D$,  $2\leq {\ell} \leq 5.$

\begin{theorem} \label{7odd} For $G=\mathbb Z_3\times \mathbb Z_6$, the odd integer group determinants $D$  with $3^7\parallel D$ are the integers of the form $3^7(2m+1)p$, $p$ a Type 1 prime.

For the even integer group determinants $D$ with $3^7\parallel D$, we achieve all values of the form $3^7\cdot 2^6 m$, $m$ in $\mathbb Z,$ and the
\be \label{niceeven} 3^7\cdot 2^4  (2m+1)(7+18k), \quad k,m\in \mathbb Z. \ee
The remaining even $D$ with $3^7\parallel D$ contain a Type 1 prime $p$. These are the
\be \label{typea}3^7\cdot 2^2 m(4+9k)p,\quad k,m \in \mathbb Z, \ee
where $p$ is any Type 1 prime,  the
\be \label{typeb} 3^7\cdot 2^2 m(2+9k)p,\quad m,k\in \mathbb Z, \ee
where $p$ is a Type 1 prime with $p\equiv 7\bmod 18$, and the
\be \label{typec}  3^7\cdot 2^2 mp,\quad m\in \mathbb Z \ee
where $p$ is a Type 1 prime with $p\equiv 13\bmod 18.$
\end{theorem}

We give the proofs of these results in Section \ref{productproofs}

\section{ \Large The integer group determinants for $\mathbb Z_{18}$ }

We shall treat this group as 
$$G=\mathbb Z_2 \times \mathbb Z_9=\langle X,Y\; : \; X^2=Y^9=1\rangle, $$
and for an
$F(X,Y)$ in $\mathbb Z[G],$ write its group determinant as a product of two $\mathbb Z_9$
group determinants
$$ D=D_1D_2,\quad D_1=\prod_{y^9=1}F(1,y),\quad  D_2=\prod_{y^9=1}F(-1,y).$$
Here we have the basic restrictions:
\be \label{CyclicRestricts} D_1\equiv D_2 \bmod 2,\quad\quad 3\mid D_i \Rightarrow 3^3\mid D_i. \ee

For the $3^{\ell}\parallel D$ with $\ell \neq 3,$ we get all values satisfying these restrictions:

\begin{theorem} \label{easycyclic} For $G=\mathbb Z_{18},$ the integer group determinants $D$ with $3^{\ell}\parallel D$, $\ell \neq 3$, are the $m$ with $\gcd(m,6)=1$, the $2^2m$ with $\gcd(m,3)=1$, the $3^4(2m+1),$ and  the $2^2\cdot 3^4m$ with any $m\in \mathbb Z$.

\end{theorem}

This leaves the $D$ with  $3^3\parallel D$. In the odd case,  they must contain at least one Type 1 prime.

\begin{theorem}\label{cyclic3}
For $G=\mathbb Z_{18},$ the odd integer group determinants $D$ with $3^3\parallel D$ take the form $3^3mp,$ with $p$ a Type 1 prime and $\gcd(m,6)=1$.

The even $D$ with $3^3\parallel D$ take one of the following forms:
the $2^2\cdot 3^3mp$ for any Type 1 prime $p$ and $\gcd(m,3)=1$, 
the  $2^4\cdot 3^3 m$ with $\gcd(m,6)=1$, and all $2^6\cdot 3^3m$ with $\gcd(m,3)=1$.
\end{theorem}

We give the proofs in Section \ref{cyclicproofs}.


\section{Proof of Theorems \ref{not7}  and \ref{7odd}} \label{productproofs}

\begin{proof}[Proof of Theorem \ref{not7}]
Notice that $F(x,y,z)=x$ has $D=-1,$ so by multiplicativity, we 
immediately obtain $\pm m$ whenever we obtain $m$.

If $3\nmid D,$ then the $D_1,D_2\equiv \pm 1 \bmod 9$. If $D$ is odd, then plainly
$D\equiv \pm 1 \bmod 18,$ and we achieve  $1+18m$ with 
$$ F(x,y,z)=1+(1+x)m\Phi(y)\Phi(z),\;\; \Phi(y):=1+y+y^2. $$
For $3\nmid D$, $D$ even, we must have $D_1=2m_1$, $D_2=2m_2$ with $m_1,m_2\equiv \pm 4\bmod 9$
and we obtain $2^2(4+9m)(4+9k)$ from
$$ -1+\left(1+m(x+1)+k(1-x)\right) \Phi(y)\Phi(z). $$

For the odd multiples of 3, we get $3^6(6m+1)$ with
$$ 1+(x+1)y+m(x+1)\Phi(y)\Phi(z), $$
and $3^8(2m+1)$ with
$$\Phi(z)+y(1+z)+x(\Phi(z)+y)+m(1+x)\Phi(y)\Phi(z).$$

For the even $3^{\ell}\parallel D$ with $6\leq \ell <12,$ we must have $D_1=2\cdot 3^{\ell} m_1$, $3\nmid m_1$ and $D_2=2m_2$, $m_2\equiv \pm 4$ mod 9 (or vice versa).
For $\ell =6,$ we get $ 2^2\cdot 3^6(3m+1)(9k+4)$ from
$$ F(x,y,z)=1+(1+x)y-( (x+1)m+(1-x)k +1)\Phi(y)\Phi(z).$$
For ${\ell} \geq 8,$ we get $2^2\cdot 3^8m(4+9k)$ using
$$F(x,y,z)=yz^2+y^2\Phi(z)-x(\Phi(z)+yz)+ (m(x+1) +k(1-x))\Phi(y)\Phi(z). $$

For the even $3^{\ell} \parallel D$ with $\ell \geq 12,$ we can have $3^6$ from both $D_1$ and $D_2,$ so do not need the factor $\pm 4 \bmod 9$. We achieve all these $3^{\ell}m$.
We get  $3^{12}\cdot 2^2(1+3m) $  from
$$ 1+2y-(1+m(1+x))\Phi(y)\Phi(z), $$
the $3^{13}\cdot 2^2(1+3m) $  from
$$ \Phi(z)(1+y)-x(1+z-yz^2-y^2z)+m(1+x)\Phi(y)\Phi(z), $$
and the $3^{14}\cdot 2^2m$ from
$$ 1+2y-y\Phi(z)+\left( (1-x) +m(x+1)\right) \Phi(y)\Phi(z). \qedhere$$
\end{proof}

\begin{proof}[Proof of Theorem \ref{7odd}]

We first deal with the odd $D$.

\subsection{Achieving all  odd multiples of a Type 1 prime $p$.}

 Write $p$ in the form \eqref{Type1} and take $F(x,y,z)$ to be
\be \label{oddtype} -xz+(1+x)\left( \Phi(z)+y(1+z)+ \Phi(z)(1-y)(Ay+B)+ m\Phi(y)\Phi(z)\right). \ee
This has $F(-1,y,z)=z$ and plainly $D_2=1$. For $z=1,$ we get
$$F(1,y,1)=5+4y+6(1-y)(Ay+B)+6m\Phi(y) $$
with $F(1,1,1)=3^2(1+2m)$. When $y=\om$ or $\om^2,$ we get
$$ F(1,y,1)=(1-y)(6B+2+y(6A+3)),\;\; F(1,\om,1)F(1,\om^2,1)=3p. $$
When $z=\om$ or $\om^2,$ we have $F(1,y,z)=-z+2y(1+z)=-z(1+2yz)$,
 $$F(1,u\om^2,\om)F(1,u\om,\om^2)=(1+2u)^2,$$ 
and these terms, $u=1,\om,\om^2$ together contribute $3^4$.

\subsection{\bf Showing that the odd $D$ must contain a Type 1 prime $p$}

Suppose that $3^7\parallel D$ with $D$ odd. Switching $x$ to $-x$ as needed we suppose that $3^7\parallel D_1$ and $3\nmid D_2$. We need to show that a Type 1 prime must divide $D$.
Expanding in powers of $y-1$ and $z-1,$ we write 
$$F(x,y,z)=f(y,z)+(1-x)g(y,z) $$
with
$$ f(y,z)=a_0+a_1(y-1)+a_2(z-1) +b_1(y-1)^2+b_2(y-1)(z-1)+b_3(z-1)^2+C(y-1,z-1),$$
where $C$ contains the remaining degree 3 and higher terms,
$$C(y-1,z-1)=c_1(y-1)^2(z-1)+c_3(y-1)(z-1)^2+c_4(y-1)^2(z-1)^2, $$ 
and
$$ g(y,z)=h_0+h_1(y-1)+h_2(z-1) + Q(y-1,z-1),$$
where $Q(y-1,z-1)$ similarly contains the degree 2 and higher terms.

Since $3\mid D_1$ and $3\nmid D_2,$ we have $F(1,1,1)=a_0=3^k\alpha_0$, $3\nmid \alpha_0$ with $k\geq 1$ and $3\nmid F(-1,1,1)= h_0$. Plainly $1-\om\mid F(1,y,z)$ for the eight other terms
$y,z=1,\om,\om^2$, $y,z\neq 1,1$, contributing an additional $3^4$ so that $k=1,2$ or $3$.

Notice that $3\mid F(1,1,z)$ iff $3\mid a_2$ and $3\mid F(1,y,1)$ iff $3\mid a_1$
and that $3$ does not divide both $a_1$ and $a_2$, else $3$ divides all the $F(1,y,z)$ and $3^9\mid D$. Moreover, if $3\nmid a_1,a_2,$ we have 
$ 3\mid a_1(\om-1) +a_2(\om -1)$ if $a_1\equiv -a_2 \bmod 3 $
and $3\mid F(1,\om,\om),F(1,\om^2,\om^2)$, 
while $ 3\mid a_1(\om-1) +a_2(\om^2 -1)=(a_1+2a_2)(\om-1) +a_2(\om-1)^2$ if $a_1\equiv a_2 \bmod 3 $ and $3\mid F(1,\om,\om^2),F(1,\om^2,\om)$.
Hence, in all cases, we pick up another 3 from the $F(1,y,z)$, $y$ or $z$ complex and $k=1$ or $2$.
Switching $y$ and $z$ if $3\mid a_1$ and doing a substitution $y\mapsto yz$ or $yz^2$
in the $3\nmid a_1a_2$ case, we can assume that $3\mid F(1,1,z)$
and write $a_2=3\alpha_2$, $3\nmid a_1$, $k=1$ or 2. 
This also forces $1-\om \parallel F(1,\om,z),F(1,\om^2,z),$ so that these six terms contribute
exactly $3^3$ and $k=1$ or $2$.

For $k=2,$ we must have 
$$3\parallel F(1,1,z)=3^2\alpha_0+ 3\alpha_2(z-1)+b_3(z-1)^2$$ 
for $z=\om$ and $\om^2$, and $3\nmid b_3$.

For $k=1,$ we must have
$$(1-\om)^3 \parallel F(1,1,z)=3\alpha_0+ 3\alpha_2(z-1)+b_3(z-1)^2=3(\alpha_0-b_3+(\alpha_2 -b_3)(z-1) )$$ 
for $z=\om$ and  $\om^2$ and $b_3\equiv \alpha_0\bmod 3$ and $b_3 \not\equiv \alpha_2\bmod 3$. 

In either case, we have $3\nmid b_3$.


\vskip0.1in
We need first to characterize the elements in $\mathbb Z[\om]$ that do not lead to Type 1 primes.

\begin{lemma}\label{evenodd} Every $\alpha$ in $\mathbb Z[\om]$ with $\gcd(N(\alpha),6)=1$ can be written in the form
$$ \alpha =  u(A+3B\om),\quad 3\nmid A, \; A,B \text{ not both even,} $$
for some unit $u= \om^j, j=0,1,$ or 2.

\end{lemma}

Hence, we can write any prime $p\equiv 1\bmod 3$ in the form $p=N(a+3b\om)$.
If $2\nmid b,$ then $p=N(a+3b\om^2)=N(a-3b-3b\om),$ and we have a representation
with $a$ even (as well as one with $a$ odd) giving \eqref{types}

We shall say that $\alpha$ with $\gcd(N(\alpha),6)=1$  is  {\em even} if $2\mid B$ and {\em odd} if $2\nmid B$. We note the following properties

\begin{lemma}\label{properties}
i) An $\alpha$ can not be both even and odd.

ii) Conjugates have the same
`parity'.

iii) If $\alpha$ and $\alpha'$ are even, then $\alpha \alpha'$ is even.
\end{lemma}

Notice that this says that Type 1 or Type 2 primes are the norms of  odd or even primes in $\mathbb Z[\om]$ respectively (and cannot be both Type 1 and Type 2).  Lemma \ref{properties} also says that an $\alpha$ in $\mathbb Z[\om]$ with $\gcd(N(\alpha),6)=1$ and $N(\alpha)$ not divisible by a Type 1 prime must be even and, since $A\equiv \pm 1 \bmod 6$,  must take the form
$$ \alpha = \pm \om^j \bmod 6. $$

\begin{proof}[Proof of Lemma \ref{evenodd}] Suppose that $\alpha=A+B\om$. If $3\mid B,$ then we are done. If $3\mid A,$ 
we take 
$$\om^2 \alpha =A\om^2+B=(B-A) -A\om. $$
If $3\nmid AB,$ then  $A\equiv B \pmod 3,$ and we take
$$ \om \alpha = A\om +B\om^2=-B+(A-B)\om. $$
We do not have $A\equiv -B \pmod 3$, else $\alpha=(A+B)+B(\om-1)$ has $3\mid N(\alpha)$.
\end{proof}

\begin{proof}[Proof of Lemma \ref{properties}]
i) If $\alpha$ had both even and odd representations then we could write
$$ \pm \om^j (1+6A_1+6B_1\om)=1+3A_2 + 3(1+2B_2)\om $$
Since $3\mid \pm \om^j-1,$ we must have $\pm \om^j=1,$ but $2\nmid 3A_2+3\om$.

\nni
ii) If $\alpha=\om^j(A+3B\om),$ then it's conjugate is $\om^{2j}(A+3B\om^2)=\om^{2j}(A-3B -3B\om)$.

\nni
iii)  Observe that
$ (A_1+3B_1\om)(A_2+3B_2\om) =A_3+3B_3\om$
with $A_3=A_1A_2-9B_1B_2,$ 
$B_3=A_1B_2+A_2B_1-3B_1B_2, $
and $B_3$ will be even if both $B_1$ and $B_2$ are. Hence the product of two evens is even.
Note,  $B_3$  will be  odd if only  one of $B_1$,$B_2$ is even, so the product of an odd and an even is odd. The product of two odds can be either even or odd.
\end{proof}

\nni
We suppose $D$ contains no Type 1 primes. Then the
$F(1,\om,z)$, $z=1,\om,\om^2$ are all of the form $(1-\om)(\pm \om^j \bmod 6),$
and the $F(-1,\om,z)$ are of the form $\pm \om^j \bmod 6$. Since $\om^2-1=2(\om-1)+(\om-1)^2,$ we get
\begin{align*}
F(1,\om,1) & =3^k\alpha_0+ a_1(\om-1) + b_1(\om-1)^2  \bmod (1-\om )^3,\\
F(1,\om,\om) & =3^k\alpha_0+ a_1(\om-1) + (b_1+b_2+b_3)(\om-1)^2 \bmod (1- \om )^3,\\
F(1,\om,\om^2) &= 3^k\alpha_0+ a_1(\om-1) + (b_1+2b_2+b_3)(\om-1)^2 \bmod (1-\om )^3.
\end{align*}
The difference between any two are divisible by $(1-\om)^2,$ so that we must have the same 
$\pm $ sign in each case. Replacing $F$ by $-F$ as necessary, we assume they are 
$(1-\om)\om^{t_1}$, $(1-\om)\om^{t_2},$ $(1-\om)\om^{t_3} \bmod 6(1-\om),$ respectively.
Similarly,
\begin{align*}F(-1,\om,1) & =(1-\om)\om^{t_1}+ 2h_0+ 2h_1(\om-1) \bmod 6, \\
F(-1,\om,\om) &= (1-\om)\om^{t_2}+ 2h_0+ (2h_1+2h_2)(\om-1) \bmod 6, \\
F(-1,\om,\om^2) &= (1-\om)\om^{t_3}+ 2h_0+ (2h_1-2h_2)(\om-1) \bmod 6, 
\end{align*}
and, since $(1-\om)$ divides the difference of any two of these, all again have the same $\pm$ sign, and we can write them $\ve \om^{s_1}, \ve \om^{s_2},\ve \om^{s_3} \bmod 6$
with $\ve=1$ or $-1$.

We break into two cases:

\vskip0.1in
\nni
{\bf (i) Suppose that $3\mid b_2$.} 

Considering $F(1,\om,\om)-F(1,\om,\om^2),$ we have  $(1-\om)^3 \mid (1-\om)(\om^{t_2}-\om^{t_3})$ and $t_3=t_2$. Since $3\nmid b_3,$ considering
$F(1,\om,1)-F(1,\om,\om),$ we have $(1-\om)^2 \parallel (1-\om)(\om^{t_1}-\om^{t_2})$ 
and $t_1\neq t_2.$

From $F(-1,\om,\om)-F(-1,\om,\om^2),$ we have $\ve(\om^{s_2}-\om^{s_3}) \equiv 4h_2(\om-1) \bmod 6,$ and  $2\mid (\om^{s_2}-\om^{s_3})$ gives $s_3=s_2$
and $3\mid h_2$. Hence from $F(-1,\om,1)-F(-1,\om,\om),$ we have 
$$ \ve (\om^{s_1}-\om^{s_2}) \equiv (1-\om)(\om^{t_1}-\om^{t_2}) \bmod 6. $$
Since $(1-\om)^2$ divides the righthand side we must have $s_1=s_3$ and the leftside is zero, but 2 does not divide the rightside.

\vskip0.1in
\nni
{\bf (ii) Suppose that $3\nmid b_2$.} 

This time we have  $t_2\neq t_3$ (since $(1-\om)^2\parallel  F(1,\om,\om)-F(1,\om,\om^2)$). We have $t_1=t_2$ if $b_2\equiv -b_3 \bmod 3$  (since $(1-\om)^3\mid  F(1,\om,\om)-F(1,\om,1)$), and $t_1=t_3$ if $b_2\equiv b_3 \bmod 3$  (since $(1-\om)^3\mid  F(1,\om,\om^2)-F(1,\om,1)$).

Hence, taking $F(-1,\om,1)-F(-1,\om,\om)$ or $F(-1,\om,1)-F(-1,\om,\om^2)$ as 
$b_2\equiv -b_3$ or $b_3 \bmod 3,$ we get $\ve(\om^{s_1}-\om^{s_2}) \equiv - 2h_2(\om-1)\bmod 6$ or  $\ve(\om^{s_1}-\om^{s_3}) \equiv 2h_2(\om-1)\bmod 6. $ Divisibility by 2 forces $s_1=s_2$ or $s_1=s_3,$ and in either case $3\mid h_2$. 
Hence $F(-1,\om,\om)-F(-1,\om,\om^2)$ gives
$$ \ve( \om^{s_2}-\om^{s_3}) \equiv (1-\om)(\om^{t_2}-\om^{t_3}) \bmod 6, \quad t_2\neq t_3. $$
But, again,  $3$ dividing the left side forces $s_2=s_3$ (and the lefthand side is zero), and we get a contradiction from 2 dividing the righthand side.

\subsection{Achieving the even values}

We will need an alternative way to write the Type 1 primes $p\equiv 4$ or 7 mod 9.
\begin{lemma} If $p$ is a Type 1 prime,  then there are integers $A,B$ with 
$$ p=\begin{cases} N(-1-3\om +6(\om-1)(A\om+B)),  & \text{ if $p\equiv 7 \bmod 9,$}\\
N(1-3\om +6(\om-1)(A\om+B)),  & \text{ if $p\equiv 4 \bmod 9$.}\end{cases}$$
\end{lemma}

\begin{proof}
We write $p=N((2+6B)+(6A+3)\om^2)=N(-1-3\om+6(A+B)+6A(\om^2-1))$. 
Writing $A+B=3r+u$ with $u=0,-1,$ or 1, we have $6\cdot 3r=6(1-\om)(2+\om)r$
and $p=N(-1-3\om+6u -6(\om -1)(A_1\om +B_1))$ with $u=0$ corresponding to $p\equiv 7$ mod 9, $u=1$ to $p\equiv 4$ mod 9, and $u=-1$ to $p\equiv 1$ mod 9.
For $u=0$ we are done. For $u=1,$ we write
$p=N(-5+3\om + 6(\om -1)(A_1\om+B_1))=N(1-3\om +6(\om -1)(A_1\om +B_1+1)).$
\end{proof}

\vskip0.1in
\nni
We achieve the $3^7\cdot 2^6m$ with
$$ F(x,y,z)= 1+y-(1+x)y^2z^2+(1+x)m\Phi(y)\Phi(z) $$
and $3^7\cdot 2^4(2m+1)(7+18k)$ from a shift of this changing only the $F(\pm 1,1,1)$
$$ 1+y-(1+x)y^2z^2-(1+(1+x)m+(1-x)k)\Phi(y)\Phi(z) $$
with $F(1,1,1)=18$ and $F(-1,1,1)=2$ becoming $-9(1+2m)$ and $-(7+18k)$.

We achieve \eqref{typea} using an adjustment of \eqref{oddtype}
$$-xz+(1+x)\left( \Phi(z)+y(1+z)+ \Phi(z)(1-y)(Ay+B)\right)-\left(1+m(1+x)+k(1-x)\right)\Phi(y)\Phi(z). $$
This has $F(-1,1,1)=-2(4+9k)$ and $F(1,1,1)=-18m$ with the other values unchanged.

To achieve the \eqref{typeb} values $3^7\cdot 2^2m(2+9k)p,$ we write $p=N(-1-3\om+6(\om-1)(A\om+B) )$ and take
$$ x+(1+x)z-z\Phi(y)-((1+x)m+(1-x)k)\Phi(y)\Phi(z)+ (1-x)\Phi(y)(z-1)(Az+B ).$$
We have $F(1,1,1)=-18m$, for $z=\om,\om^2,$ we have $F(1,1,z)=1-z$ contributing 3, 
while for $y=\om,\om^2,$ we have $F(1,y,z)=1+2z$ together contributing $3^4$. So $D_1=-3^7\cdot 2 m$.  We have $F(-1,y,z)=-1-z\Phi(y)-2k\Phi(y)\Phi(z)+2\Phi(y)(z-1)(Az+B)$.
The terms $y=\om,\om^2$ have $F(-1,y,z)=-1$, so contribute 1, while
$F(-1,1,1)=-2(2+9k)$. The terms with $z=\om,\om^2$ have
$ F(-1,1,z)=-1-3z +6(z-1)(Az+B)$ contributing $p$. Hence, $D_2=-2p(2+9k).$

For the \eqref{typec}, we write
$p=N(1-3\om+6(\om-1)(A\om +B) )$ and achieve $3^7\cdot 2^2 mp$ with
$$F=1+(1+x)z-z\Phi(y)-(1+x)m\Phi(y)\Phi(z)+ (1-x)\Phi(y)(z-1)(Az+B ).$$
This has $F(-1,1,1)=-2$, for $z=\om,\om^2$ has $F(-1,1,z)=1-3z+6(z-1)(Az+B)$ contributing $p$, and for $y=\om,\om^2$ has $F(-1,y,z)=1$. So, $D_2=-2p$. 
It has $F(1,y,z)=1+2z-z\Phi(y)-2m\Phi(y)\Phi(z)$ and $F(1,1,1)=-18m$, when $z=\om,\om^2$ the $F(1,1,z)=1-z$ contributing $3$, when $y=\om,\om^2,$ the $F(1,y,z)=1+2z$ together contributing $3^4$. Hence, $D_1=-3^7\cdot 2m.$

\subsection{Showing that the even values take the stated forms}
Suppose now that we have a $D$ with $3^7\parallel D$ and $2^j\parallel D$, $2\leq j\leq 5$. We assume that $3\mid D_1$. Since 2 stays prime in $\mathbb Z[\om]$, we know that if a 2 comes from a term $F(x,y,z)$ with $y$ or $z$ complex, then $2\mid F(\pm 1,y,z), F(\pm 1,\overline{y},\overline{z})$ contributing a $2^4$ to $D$. In particular we can have no other twos (if $2^2\mid F(x,y,z),$ then $2^2\mid F(x,\overline{y},\overline{z})$ and $2^6\mid D$), and we have $D_1=3^7\cdot 2^2(2m+1)$, $D_2=2^2t$ with $t\equiv \pm 2 \bmod 9$ odd. These are the \eqref{niceeven}.

That just leaves all the $2^j\parallel D,$ $2\leq j\leq 5$, with $2\mid F(\pm 1,1,1)$
and none of the other terms. In the proof of Theorem \ref{7odd}, we showed that the $3^l\parallel D$ have $3$ or $3^2\parallel F(1,1,1)$, so if we take $F(x,y,z)-3\Phi(y)\Phi(z),$ then this will still have $3^7\parallel D$, the complex terms are unchanged,  but now  $D$ is odd. We showed in the proof of Theorem \ref{7odd} that the complex terms in an odd $D$ with $3^7\parallel D$ must contribute  a Type 1 prime $p$, so the same was true for our original $F$. Assume $3\mid D_1$. 
If the Type 1 prime $p\mid D_1,$ then $D_1=3^7\cdot 2mp$, $D_2=2t$, $t\equiv \pm 4$ mod 9; that is,  it must be  of the form \eqref{typea}.

If the Type 1 prime  $p\mid D_2,$ then we have $D_1=3^7\cdot 2m$, $D_2=2\cdot tp$ with $pt\equiv \pm 4$ mod 9.  
If $p\equiv 7$ mod 9, then $t\equiv \pm 2$ mod 9, giving us type \eqref{typeb}. 
The  $p\equiv 4$ mod 9 are covered by \eqref{typec}. 
If $p\equiv 1$ mod 9, then $t\equiv \pm 4$ mod 9 and these are already covered in \eqref{typea}.
\end{proof}

\section{Proof of Theorems \ref{easycyclic} and \ref{cyclic3}} \label{cyclicproofs}

In this section,  we will write $\om=e^{2\pi i/3}$ and $\om_9=e^{2\pi i /9}$.
We write 
$$N(f(\om))=f(\om)f(\om^2),\quad N_9(f(\om_9))=f(\om_9)f(\om_9^2)f(\om_9^4)f(\om_9^5)f(\om_9^7)f(\om_9^8),$$
 for the norms from $\mathbb Z[\om]$ and $\mathbb Z[\om_9]$ to $\mathbb Z,$
noting that $f(\om_9)f(\om_9^4)f(\om_9^7)$ and $f(\om_9^2)f(\om_9^5)f(\om_9^8)$
will be in $\mathbb Z[\om]$.

\begin{proof}[Proof of Theorem \ref{easycyclic}] The $m$ with $\gcd(m,18)=1$ and the $18^2m,$ any $m,$ follow from Newman's classical  result for the cyclic case \eqref{Newman}.

We get $2^2$ and $2^3$ from
$ 1+y$ and $1+(1+x)y+y^2.$
Multiplicativity then gives us all the $2^km$, $\gcd(m,6)=1$, $k\geq 2$.

We get $3^4(2m+1)$ with
$$ (1+x) + y^2+xy^3+y^4+xy^5+y^6+xy^7+y^8+m(1+x)(y^9-1)/(y-1). \qedhere$$
\end{proof}

\begin{proof}[Proof of Theorem \ref{cyclic3}] We first show how we achieve the values. Writing our Type 1 prime $p$ in the form $p=N(2(1+3B) +3(2A+1)\om),$
then
$$ 1+y+y^2-y^5-y^8+x(1-y-y^2+y^4+y^6+y^7)
+(1+x)(Ay+B)(y-1)(1+y^3+y^6)$$
has $D_1=3^3p, D_2=1$ and $D=3^3p$.

We get $2^4\cdot 3^3$ from
$$ 1+y+y^2-y^8+x(y+y^3+y^6-y^7-y^8).$$

The other values follow from multiplicativity, since in Theorem \ref{easycyclic} we obtained
all $\gcd(m,6)=1$ and all $2^k$ with $k\geq 2$.

\vskip0.1in
It remains to show that a determinants must be of the stated form; we separate into $D$ odd and $D$ even.

\vskip0.1in
\nni
{\bf $D$ is odd}
 
In this case, we need to show that the odd $D$ with $3^3\parallel D$ must contain
a Type 1 prime $p$. We assume that $3\mid D_1$ so that $3\parallel F(1,1), N(F(1,\om)), N_9(F(1,\om_9)$ and  $3\nmid  F(-1,1), N(F(-1,\om)), N_9(F(-1,\om_9)$. 
We suppose that $N(F(\pm 1 ,\om))$ and $N_9(F(\pm 1, \om_9)$ do not contain any Type 1 primes. Then, in $\mathbb Z[\om]$
\begin{align*} F(1,\om) & =\pm \om^{t_1}(\om-1)(1+6A_1+6B_1\om), \\
F(1,\om_9)F(1,\om_9^4)F(1,\om_9^7) & =\pm \om^{t_2}(\om-1)(1+6A_2+6B_2\om),\\
F(-1,\om) & =\pm \om^{t_3}(1+6A_3+6B_3\om), \\
F(-1,\om_9)F(-1,\om_9^4)F(-1,\om_9^7)& = \pm \om^{t_4}(1+6A_4+6B_4\om),
\end{align*}
for some $t_i=0,1$ or 2. Multiplying $F(x,y)$ by $\pm y^j$ we assume that
$$ F(1,\om)=(\om -1) (1+6A_1+6B_1\om) $$
and write
$F(x,y)=f(y)+(1-x)g(y)$ with
$$f(y) = (y-1)(1 + 6A_1+6B_1y) + (y^2+y+1)h(y),\quad h(y)=h_0+h_1(y-1)+\cdots +h_6(y-1)^6. $$
Since $3\parallel F(1,1)=3h_0$ we have $h_0\equiv \pm 1\bmod 3$.
Taking 
$$F_1(y)=y-1+h_0(y^2+y+1)+ (y^3-1)(h_1+h_2(y-1)), $$ 
and checking  on Mathematica (note $3\mid h_1^3-h_1$), we see that
$$ F_1(y)F_1(y^4)F_1(y^7)\equiv (y^3-1) (1+3h_0(h_0-1))-3h_0^3y^3   +9s_1(y^3)   \bmod y^6+y^3+1.$$
Observing that divisibility by $(\om_9-1)^7$ means divisibility by $3(\om_9-1)$ and hence in $\mathbb Z[\om]$ by $3(\om -1)$, and $y^2+y+1=(y-1)^2+3y$,  we get
$$\pm \om^{t_2}(\om -1) \equiv (\om -1)\pm 3 \;\; \bmod 3(\om -1), $$
and deduce that $\pm \om^{t_2}=\om$ or $\om^2$.

Writing
$$ g(y)=b_0+b_1(y-1)+\cdots + b_8(y-1)^8, $$
we have $F(-1,1)=3h_0+ 2b_0$ and $b_0\equiv \ve \mod 3$ with $\ve =1$ or $-1$.  We get
$$ \pm \om^{t_3} \equiv (\om-1) + 2\ve +2b_1(\om-1) \bmod 6. $$
Checking the possibilities we must have $\pm \om^{t_3}=-\ve \om^{2}$  and
$b_1\equiv 0\bmod 3$ when $b_0\equiv 1\bmod 3$ and $b_1\equiv -1\bmod 3$ when $b_0\equiv -1\bmod 3$. So, we have
$$F(-1,y)=F(1,y)+ 2\lambda(y) \bmod 6,  \quad F(1,y)=(y-1)+(y^2+y+1)h(y)\bmod 6 $$
with the two possibilities
$$ \lambda (y) =1 +b_2(y-1)^2+\cdots +b_8(y-1)^8  \text{ or } -y+b_2(y-1)^2+\cdots +b_8(y-1)^8. $$
We write 
$$ \pm \om^{t_4}\equiv  A+2B+4C+8E \quad \bmod 6, $$
where
\begin{align*}
A & = F(1,\om_9)F(1,\om_9^4)F(1,\om_9^7) \equiv \om (\om-1) \text{ or } \om^2(\om-1) \bmod 6, \\
B & =F(1,\om_9)F(1,\om_9^7)\lambda(\om_9^4)+ F(1,\om_9)F(1,\om_9^4)\lambda(\om_9^7)+F(1,\om_9^4)F(1,\om_9^7)\lambda(\om_9),\\
C & =F(1,\om_9)\lambda(\om_9^7)\lambda(\om_9^4)+ F(1,\om_9^7)\lambda(\om_9^4)\lambda(\om_9)+F(1,\om_9^4)\lambda(\om_9^7)\lambda(\om_9),\\
E & =\lambda(\om_9)\lambda (\om^4) \lambda (\om^7).
\end{align*}

Observing that working mod 6, we can ignore the $b_4(y-1)^4$ and higher terms (since
$(1-\om_9)^4$ gives a multiple of $(1-\om)(1-\om_9)$ and so in $\mathbb Z[\om]$ 
gives at least a 3), one can check that
$$ \lambda_1(y)=1+b_2(y-1)^2+b_3(y-1)^3 \text{ or } -y+b_2(y-1)^2+b_3(y-1)^3$$
has 
$$\lambda_1(y)\lambda_1(y^4)\lambda_1(y^7)\equiv 1 \text{ or } -y^3 \quad \bmod y^6+y^3+1 $$
and
$$ 8E\equiv 2 \text{ or } -2\om \quad \bmod 6. $$
For $C \bmod 3,$ we can replace $F(1,y)$  and $\lambda(y)$ by 
$$ F_2(y)=(y-1)+h_0(y^2+y+1)+h_1(y^3-1),\;\; \lambda_2(y)=1 +b_2(y-1)^2 \text{ or } -y + b_2(y-1)^2, $$
and check that
$$F_2(y)\lambda_2(y^7)\lambda_2(y^4)+ F_2(y^7)\lambda_2(y^4)\lambda_2(y)+F_2(y^4)\lambda_2(y^7)\lambda_2(y)\equiv 3s_2(y) \bmod y^6+y^3+1,$$
giving $4C\equiv 0 \bmod 6.$

For $B$ mod 3, we can replace $F(1,y)$ and $\lambda(y)$ by
$$ F_3(y)=(y-1)+h_0(y^2+y+1), \quad \lambda_3(y)=1 \text{ or } -y,$$ 
and verify that
$$F_3(y)F_3(y^7)\lambda_3(y^4)+ F_3(y)F_3(y^4)\lambda_3(y^7)+F_3(y^4)F_3(y^7)\lambda_3(y)\equiv 3s_3(y) \bmod y^6+y^3+1,  $$
and $2B\equiv 0\bmod 6$.

Hence we get either
$$ \pm \om^{t_4}\equiv \om^{ t_2}(\om-1)+ 2 \bmod 6.$$
or 
$$ \pm \om^{t_4}\equiv  \om^{ t_2}(\om-1) -2\om \bmod 6,$$
with $t_2=1$ or 2. In the first case, from $\om-1 \mid 1\pm \om^{t_4},$ we have $\pm \om^{t_4}=-\om^{t_4}$, with mod 3 giving $t_4=2$, and a contradiction mod 2,  and in the second $1-\om \mid -\om\pm \om^{t_4}$ gives $\pm \om^{t_4}=\om^{t_4}$, with mod 3 giving $t_4=2$ and a contradiction mod 2.
\vskip0.2in
\nni
{\bf $D$ is even}

Notice that $2$ remains prime in the UFDs $\mathbb Z[\om]$ and $\mathbb Z[\om_9]$ (see, for example, Washington \cite{washington}, Theorems 2.13 and 11.1). Hence, if $2\mid N(F(\ve,\om)),$ then we must have $2^k\parallel F(\ve,\om^j)$, $j=1,2$, some $k\geq 1$  and, since $N(F(\ve,\om)) \equiv N(F(-\ve,\om))\bmod 2,$ also  $2^{\ell} \mid F(-\ve,\om^j))$, $j=1,2$, some $\ell\geq 1$, and these four terms contribute $2^{2t}$, some $t\geq 2$. 
Likewise, $2$ stays prime in  $\mathbb Z[\om_9],$ and  if $2\mid N_9(F(\ve,\om_9)),$
the corresponding 12 terms $F(\pm 1,\om_9^j)$, $\gcd(j,3)=1$, contribute $2^{6t},t\geq 2$.
If, in addition, the $F(\pm 1,1)$ are even then we get at least two more twos.
Hence, in all such cases, we either have $2^4\parallel D$ or $2^6\mid D$. We obtained all these.

That leaves the case $N(F(\pm1,\om)),N_9(F(\pm 1,\om_9)$ odd and $F(\pm 1,1)$
even. We must have $3\parallel F(1,1), N(F(1,\om)), N_9(F(1,\om_9))$, $3\nmid F(-1,1),$
and hence $G(x,y)=F(x,y)-(y^9-1)/(y-1)$ will just change the values $F(\pm 1,1)$ by 9 and produce an odd $D$ with $3^3\parallel D$.  This, we showed,  must have a Type 1 prime $p$
coming from $N(F(\pm 1,\om)$ or $N_9(F(\pm 1,\om_9)$ (these have not changed), and so our original $D$ contained this $p$. We obtained all the $2^2\cdot 3^3 mp$, $\gcd(m,3)=1$.
\end{proof}


\begin{thebibliography}{99}

\bibitem{Humb}
H.\ Bautista-Serrano, B.\ Paudel and C.\ Pinner,  \textit{The integer group determinants do not determine the group,}  Combinatorics \& Number Theory (formerly Moscow J. Comb. Number Theory), \textbf{13} (2024), no. 1, 59--65.

\bibitem{Humb2}
H.\ Bautista-Serrano, B.\ Paudel and C.\ Pinner,
\textit{The integer group determinants for the semidihedral group of order 16,} arXiv:2304.04379 [math.NT].


\bibitem{dihedral}
T. Boerkoel and C. Pinner, \textit{Minimal  group determinants and the  Lind-Lehmer problem for dihedral groups},   Acta Arith. \textbf{186} (2018), no. 4, 377-395.	arXiv:1802.07336 [math.NT].



\bibitem{Conrad}
K.\ Conrad, \textit{The origin of representation theory},  Enseign. Math. (2) \textbf{44} (1998), no. 3-4, 361-392.


\bibitem{Frob} F.\ G.\ Frobenius, \textit{\"{U}ber die  Primefactoren  der  Gruppendeterminante},  Gesammelte Ahhand-lungen, Band III, Springer, New York, 1968, pp. 38–77. MR0235974


\bibitem{book}
K.\ Johnson, Group Matrices Group Determinants and Representation Theory, Lecture Notes in Mathematics 2233,  Springer 2019.


\bibitem{Laquer}
H. T. Laquer, \textit{ Values of circulants with integer entries,}  in A Collection of Manuscripts Related
to the Fibonacci Sequence, Fibonacci Assoc., Santa Clara, 1980, pp. 212–217. MR0624127.



\bibitem{Mahoney}
M.\ Mahoney, \textit{Determinants of integral group matrices for some non-abelian
2-generator groups}, Linear and Multilinear Algebra, \textbf{11} (1982), no 2. 189-201.

\bibitem{Newman1}
M.\ Newman, \textit{ On a problem suggested by Olga Taussky-Todd}, Ill. J. Math. \textbf{24} (1980), 156-158.

\bibitem{Newman2}
M.\ Newman, \textit{Determinants of circulants of prime power order}, Linear Multilinear Algebra \textbf{ 9}
(1980), no. 3, 187–191. MR0601702.





\bibitem{bishnu1}
B.\ Paudel and C.\ Pinner, \textit{Integer circulant determinants of order 15,} Integers \textbf{22} (2022), Paper No. A4.

\bibitem{ZnxH}
B.\ Paudel and C.\ Pinner, \textit{The group determinants for $\mathbb Z_n \times  H$,} 
Notes on Number Theory and Discrete Mathematics, \textbf{29}  (2023), no. 3, 603--619.  

\bibitem{Q16}
B.\ Paudel and C.\ Pinner, \textit{The integer group determinants for $Q_{16}$,}
arXiv:2302.11688 [math.NT].


\bibitem{nonabelian}
B.\ Paudel and C.\ Pinner, \textit{The integer group determinants for the non-abelian groups of order 18}, arXiv 2305.01833[math.NT].



\bibitem{smallgps}
C. Pinner and C. Smyth, \textit{Integer group determinants for small groups},  Ramanujan J. \textbf{51} (2020), no. 2, 421-453.

\bibitem{TausskyTodd}
O. Taussky Todd, \textit{Integral group matrices}, Notices Amer. Math. Soc. 24 (1977), no. 3, A-345.
Abstract no. 746-A15, 746th Meeting, Hayward, CA, Apr. 22–23, 1977.

\bibitem{washington}
L.\ Washington, Introduction to Cyclotomic Fields, Springer 1982.

\bibitem{Yamaguchi1}
Y.\ Yamaguchi and N.\ Yamaguchi, 
\textit{Generalized Dedekind’s theorem and its application to integer group determinants}, J. Math. Soc. Japan \textbf{76}(4) (2024), 1123--1138.  https://doi.org/10.2969/jmsj/90539053.


 
\bibitem{Yamaguchi2}
Y.\ Yamaguchi and N.\ Yamaguchi, 
\textit{Integer circulant determinants of order 16},  Ramanujan J., \textbf{61} (2023), 1283--1294. https://doi.org/10.1007/s11139-022-00599-9.


\bibitem{Yamaguchi3}
Y.\ Yamaguchi and N.\ Yamaguchi, 
\textit{Integer group determinants for $C_2^4$}, 
 Ramanujan J. \textbf{ 62} (2023), 983--995 (2023). 
https://doi.org/10.1007/s11139-023-00727-z.



\bibitem{Yamaguchi4}
Y.\ Yamaguchi and N.\ Yamaguchi, \textit{Integer group determinants for $C_4^2$}, Integers \textbf{24} (2024) \#A30.

\bibitem{Yamaguchi5}
Y.\ Yamaguchi and N.\ Yamaguchi, \textit{Integer group determinants for abelian groups of order 16}, Hiroshima Math. J. \textbf{54} (3) (2024), 359 --373. https://doi.org/10.32917/h2023012. 





\bibitem{Yamaguchi6}
Y.\ Yamaguchi and N.\ Yamaguchi, \textit{Integer group determinants for three of the non-abelian groups of order 16}, Res. number theory \textbf{10}, 23 (2024). 
 https://doi.org/10.1007/s40993-024-00508-7.









\bibitem{Yamaguchi9}
Y.\ Yamaguchi and N.\ Yamaguchi, \textit{Integer group determinants of order 16},
 Ramanujan J. \textbf{65} (2024), 1459--1474. https://doi.org/10.1007/s11139-024-00946-y.

\end{thebibliography}
\end{document}